\RequirePackage[l2tabu, orthodox]{nag}

\documentclass[12pt]{amsart}
\usepackage{fullpage,url,amssymb,enumitem,colonequals}
\usepackage{amsfonts}
\usepackage[foot]{amsaddr}
\setenumerate{listparindent=\parindent}
\setlist[enumerate]{label={\upshape(\arabic*)}}
\usepackage{tabularx} 
\usepackage{booktabs} 
\usepackage{ragged2e}  
\newcolumntype{Y}{>{\RaggedRight\arraybackslash}X} 

\usepackage[OT2,T1]{fontenc}

\DeclareSymbolFont{cyrletters}{OT2}{wncyr}{m}{n}
\DeclareMathSymbol{\Sha}{\mathalpha}{cyrletters}{"58}

\usepackage{color}

\newcommand{\F}{\mathbb{F}}

\newcommand{\Q}{\mathbb{Q}}
\newcommand{\R}{\mathbb{R}}
\newcommand{\Z}{\mathbb{Z}}
\newcommand{\Qbar}{{\overline{\Q}}}

\newcommand{\calA}{\mathcal{A}}

\newcommand{\calW}{\mathcal{W}}

\DeclareMathOperator{\Frob}{Frob}
\DeclareMathOperator{\Gal}{Gal}

\DeclareMathOperator{\Norm}{Norm}

\newcommand{\GalQ}{{\Gal}(\Qbar/\Q)}

\newcommand{\slantsf}[1]{\textsl{\textsf{#1}}}

\newtheorem{theorem}{Theorem}[section]
\newtheorem{lemma}[theorem]{Lemma}
\newtheorem{corollary}[theorem]{Corollary}
\newtheorem{proposition}[theorem]{Proposition}

\theoremstyle{definition}
\newtheorem{definition}[theorem]{Definition}

\newtheorem{conjecture}[theorem]{Conjecture}

\theoremstyle{remark}

\usepackage{microtype}

\usepackage[
	backref,
	pdfauthor={Borys Kadets}, 
]{hyperref}
\usepackage[alphabetic,backrefs,lite]{amsrefs} 

\begin{document}

\title[On the number of rational points on simple abelian varieties over finite fields]{Estimates for the number of rational points on simple abelian varieties over finite fields}
\author{Borys Kadets}
\thanks{This research was supported in part by Simons Foundation grant \#550033.}
\address{Department of Mathematics, Massachusetts Institute of Technology, Cambridge, MA 02139-4307, USA}
\email{bkadets@math.mit.edu}
\urladdr{\url{http://math.mit.edu/~bkadets/}}

\begin{abstract}
Let $A$ be a simple abelian variety of dimension $g$ over the field $\F_q$. The paper provides improvements on the Weil estimates for the size of $A(\F_q)$. For an arbitrary value of $q$ we prove $(\lfloor(\sqrt{q}-1)^2 \rfloor + 1)^g \leqslant \#A(\F_q) \leqslant (\lceil(\sqrt{q}+1)^2 \rceil - 1)^{g}$ holds with finitely many exceptions. We compute improved bounds for various small values of $q$. For instance, the Weil bounds for $q=3,4$ give a trivial estimate $\#A(\F_q) \geqslant 1$; we prove $\# A(\F_3) \geqslant 1.359^g$ and $\# A(\F_4) \geqslant 2.275^g$ hold with finitely many exceptions.  We use these results to describe all abelian varieties over finite fields that have no new points in some finite field extension.
\end{abstract}

\maketitle

\section{Introduction}\label{S:introduction}
Let $A$ be an abelian variety of dimension $g$  defined over a finite field $\F_q$. The following classical theorem of Weil gives an estimate for the size of the group $A(\F_q)$.

\begin{theorem}[Weil \cite{Weil1948-RH}]\label{Weil}
Suppose $A/\F_q$ is an abelian variety of dimension $g$. Then
\[ (\sqrt{q}-1)^{2g} \leqslant \#A(\F_q) \leqslant (\sqrt{q}+1)^{2g}.\]
\end{theorem}

Our goal is to improve upon the estimates of Theorem \ref{Weil}.
For example, note that the lower bound is vacuous for $q=2,3,4$; our results imply an exponential lower bound in the cases $q=3,4$, while for $q=2$ there are infinitely many abelian varieties with one point as proved in \cite{Madan-Pal1977}.

By Poincar\'{e} reducibility theorem (see \cite[Theorem~1,~Section~19]{MumfordAV1970}) any abelian variety is isogenous to a product of simple abelian varieties. Since isogenies preserve point counts, it is natural to consider only simple abelian varieties. Let $\mathcal{A}_q(g)$ denote the (finite) set of isogeny classes of simple abelian varieties of dimension $g$ over $\F_q$. Let $\mathcal{A}_q$ denote the union $\mathcal{A}_q\colonequals \bigsqcup_g \mathcal{A}_q(g)$. Define the quantities $a(q), A(q)$ by the formulas
\[a(q)\colonequals \liminf_{A \in \mathcal{A}_q} \#A(\F_q)^{1/g},\ \ A(q)\colonequals \limsup_{A \in \mathcal{A}_q} \#A(\F_q)^{1/g}. \]

Serre (see \cite{SerreN_x(p)}~{Section~4.6}) noticed that for general varieties the estimates coming from the Weil conjectures can be improved using some metric properties of totally positive algebraic integers. In the case of abelian varieties, Aubry, Haloui and Lachaud \cite{Aubry-et-al2013} observed that the asymptotic behavior of $A(q)$ is related to the Schur-Siegel-Smyth trace problem. Let us briefly recall its statement.

\begin{definition}
	Suppose $\alpha$ is an algebraic integer. Let $\alpha=\alpha_1, \alpha_2, ..., \alpha_n$ denote the Galois orbit of $\alpha$. The \slantsf{normalized trace} of $\alpha$ is the average value of its Galois conjugates: $\mathrm{tr}(\alpha)\colonequals 1/n (\alpha_1 +\dots +\alpha_n).$
\end{definition}
\begin{definition}\label{SSS}
	Let $\mathrm{TP} \subset \Qbar$ denote the set of all totally positive algebraic integers. The \slantsf{Schur-Siegel-Smyth} constant $\rho$ is defined by \[\rho \colonequals \liminf_{\alpha \in \mathrm{TP}} \mathrm{tr}(\alpha).\]
\end{definition}

The following conjecture is known as the Schur-Siegel-Smyth trace problem; see \cite{Borwein2002}*{Chapter~10}.
 
 \begin{conjecture}\label{trace-problem}
 With notation as above $\rho=2$.
 \end{conjecture}

Suitably modified Chebyshev polynomials give an infinite family of totally positive algebraic integers with normalized trace $2$, so $\rho \leqslant 2$. The current best lower bound for $\rho$ is 1.79193, see \cite{Liang-Wu2011}.

The following proposition is implicit in \cite{Aubry-et-al2013}.
\begin{proposition}\label{rho-limit}
	We have
	\[\lim_{q \to \infty} (q+1)^2-A(q^2) \geqslant \rho.\]
	For every $q$ we have $(q+1)^2\leqslant A(q^2)+2+q^{-2}$.
\end{proposition}

Assuming Conjecture \ref{trace-problem}, from Proposition \ref{rho-limit} we get $A(q^2)=(q+1)^2-2+o(1)$ as $q \to \infty$. Even though the exact values of $a(q), A(q)$ seem hard to determine, the following proposition shows that they are not far from the Weil bounds.

\begin{proposition}\label{Chebyshef}
For every prime power $q$ the following inequalities hold:
\[(\sqrt{q}-1)^2 \leqslant a(q) \leqslant \lceil(\sqrt{q}-1)^2\rceil + 2,  \]
\[\lfloor(\sqrt{q}+1)^2\rfloor - 2 - q^{-1} \leqslant A(q) \leqslant (\sqrt{q}+1)^2.\]
\end{proposition}

Even though the Weil bounds cannot be significantly strengthened for large values of $q$, it is still interesting to get some improvements. The following theorem of Aubry, Haloui and Lachaud gives one such improvement.

\begin{theorem}{\cite{Aubry-et-al2013}*{Corollaries~2.2~and~2.14}}\label{AHL}
For any prime power $q$ the following inequalities hold:
\[	a(q) \geqslant \lceil(\sqrt{q}-1)^2 \rceil, \ \ A(q) \leqslant \lfloor(\sqrt{q}+1)^2 \rfloor . \]                                             
\end{theorem}

We will derive improved estimates for $a(q)$ and $A(q)$ without using metric properties of traces and focusing on the case of small values of $q$. As a demonstration of the method for arbitrary $q$ we give a simple proof of a stronger version of Theorem \ref{AHL}.

\begin{theorem}\label{easy-bounds}
	For any prime power $q$ the following inequalities hold:
	\[	a(q) \geqslant \lfloor(\sqrt{q}-1)^2 \rfloor + 1, \ \ A(q) \leqslant \lceil(\sqrt{q}+1)^2 \rceil - 1 . \]                                             
\end{theorem}

Theorem \ref{easy-bounds} is equivalent to Theorem \ref{AHL} when $q$ is not a square.  For small values of $q$ we obtain the following result.

\begin{theorem}
For $q=2,3,4,5,7,8,9$ the upper and lower bounds on $a(q)$ and $A(q)$ are given in Table \ref{intro:bounds}.
	\begin{table}[h]
	\begin{center}
		\begin{tabular*}{94pt}{c|cc}
			\toprule
			q  & a(q) & A(q)\\
			\midrule
			2 & 1 & 4.035\\
			3 & 1.359 &5.634\\
			4 & 2.275& 7.382\\
			5 & 2.7 &8.835\\
			7 & 3.978 & 11.734\\
			8 & 4.635& 13.05\\
			9 & 5.47& 14.303\\
			\bottomrule
		\end{tabular*}\caption{Lower and upper bounds on $a(q)$ and $A(q)$, respectively.}\label{intro:bounds} 
	\end{center}
	
\end{table}
\end{theorem}

Madan and S\"at \cite{Madan-Pal1977} give an explicit list of all isogeny classes of simple abelian varieties over $\F_2$ with $\#A(\F_2)=1$. We do not know if for some $k>1$ there are infinitely many simple abelian varieties with $\#A(\F_2)=k$.

\section{Abelian varieties over large fields}
We use an explicit description of the set $\mathcal{A}_q$ of isogeny classes of simple abelian varieties provided by the Honda--Tate correspondence (see, for example \cite{Waterhouse1969}). Recall that by Honda-Tate, the set $\mathcal{A}_q$ is in bijection with the set $\calW_q$ of Galois orbits of $q$-Weil numbers. For our purposes it is more convenient to use an equivalent description of the set $\calA_q$ in terms of certain totally real algebraic integers. 

\begin{proposition}\label{Honda-Tate}
	Let $\mathcal{A}_q'$ denote the set of Galois orbits of totally real algebraic integers $\alpha$ such that $\alpha$ and all of its Galois conjugates lie in the interval $\left[ (\sqrt{q}-1)^2, (\sqrt{q}+1)^2 \right]$. Then there is a bijection $\phi: \mathcal{A}_q' \to \mathcal{W}_q$ such that for every $\alpha\in \Qbar$, with $\GalQ \alpha \in \mathcal{A}_q'$ the abelian variety $A$ corresponding to $\phi(\GalQ\alpha)$ under the Honda-Tate correspondence satisfies
	\[\left(\Norm \alpha \right)^{1/\deg \alpha} = \Big(\#A(\F_q)\Big)^{1/\dim A}\]
\end{proposition}
\begin{proof}
	Given a $q$-Weil number $\gamma$, define the algebraic integer $\alpha$ by $\alpha \colonequals (1-\gamma)(1-\overline{\gamma})=1+q-\gamma-q/\gamma$. The integer $\alpha$ is totally positive. Since $|\gamma|=\sqrt{q}$ and $\alpha = 1 + q + 2\mathrm{Re}(\gamma)$, we conclude that $\alpha$ and all of its conjugates belong to the segment $\left[1+q-2\sqrt{q}, 1+q+2\sqrt{q}\right]$.
	
	Let $\psi\colon \mathcal{W}_q \to \mathcal{A}_q'$ be the map $\gamma \mapsto (1-\gamma)(1-\overline{\gamma})$. We claim that $\psi$ is a bijection. Given $\alpha \in \Qbar$, with $\GalQ \alpha \in \mathcal{A}_q',$ a root of $x+q/x = 1+q-\alpha$ is a $q$-Weil number (the discriminant of this quadratic equation is $(\alpha-1-q)^2-4q \leqslant 0$, and the product of its roots is $q$). This defines a map $\phi: \mathcal{A}'_q \to \mathcal{W}_q$ which is inverse to $\psi$. 
	
	If $A$ is an abelian variety corresponding to the $q$-Weil number $\gamma$, and $\alpha\in \phi(\GalQ\gamma)$, then $\#A(\F_q)^{1/2g}=\left(\Norm (1-\gamma) \right)^{1/\deg \gamma} = \left(\Norm \alpha \right) ^{1/2 \deg \alpha}$.
\end{proof}

Proposition \ref{Honda-Tate} shows that we need to understand the possibilities for the norm of a totally real algebraic integer whose conjugates lie in the interval $[(\sqrt{q}-1)^2,(\sqrt{q}+1)^2]$. The trivial inequalities $(\sqrt{q}-1)^{2\deg \alpha}\leqslant \Norm(\alpha) \leqslant (\sqrt{q}+1)^{2\deg \alpha}$ are equivalent to the Weil estimates under the correspondence of Proposition \ref{Honda-Tate}. We will produce totally real integers with almost extremal norms in Proposition \ref{Chebyshev} by utilizing shifted Chebyshev polynomials of Lemma \ref{Chebyshev-lemma}; Proposition \ref{Chebyshev} combined with Theorem \ref{Weil} gives Proposition \ref{Chebyshef}.

\begin{lemma}\label{Chebyshev-lemma}
	Let $T_n$ denote the Chebyshev polynomial of degree $n \geqslant 1$, and let $P_n$ be the integer monic polynomial defined by $P_n\colonequals T_n(x/2 - 1)$. Then for $N > 1$ we have:
	\[N+2 - 1/N \leqslant \lim_{n \to \infty}|P_n(-N)|^{1/n} \leqslant N+2.\] 
\end{lemma}
\begin{proof}
	We use the formula 	
	\[T_n(y)=\frac{1}{2}\left(\left(y+\sqrt{y^2-1}\right)^n+\left(y-\sqrt{y^2-1}\right)^n\right),\] 
	see \cite[Equation~1.49]{Mason-Handscomb2002}.
	By calculus, for $M>1$ we have $\lim_{n \to \infty} |T_n(-M)|^{1/n}=M+\sqrt{M^2-1}$. Substituting $M=1+N/2$ gives an explicit formula for $\lim_{n \to \infty}|P_n(-N)|^{1/n}$, and the conclusion follows.
\end{proof}

\begin{proposition}\label{Chebyshev}
	The numbers $a(q)$, $A(q)$ satisfy \[a(q) \leqslant \lceil(\sqrt{q}-1)^2 \rceil+2,\ \ \ A(q) \geqslant \lfloor(\sqrt{q}+1)^2 \rfloor-2-q^{-1}.\]
\end{proposition}
\begin{proof}
	To prove either inequality it suffices to construct infinitely many algebraic integers  $\alpha \in \calA_q'$ with geometric mean of the conjugates of $\alpha$ close to the corresponding end of the interval $\left[(1-\sqrt{q})^2, (1+\sqrt{q})^2\right]$.
	Let $N$ be a positive integer. For a prime number $\ell$ let $P_\ell$ be the polynomial of Lemma \ref{Chebyshev-lemma}. Recall that the roots of $T_\ell$ belong to the segment $[-1,1]$, so the roots of $P_\ell$ belong to $[0,4]$. By irreducibility of the cyclotomic polynomial, the monic integer polynomial $P_\ell$ factors as $P_\ell=(x-2)R_\ell$, where $R_\ell$ is irreducible. Let $\alpha_\ell$ denote a root of $R_\ell(x-N)$, then $\alpha_\ell$ and all of its conjugates belong to the interval $[N, N+4]$. The norm of $\alpha_\ell$ satisfies
	\[\lim_{\ell \to \infty} \Norm(\alpha_\ell)^{1/\deg \alpha_\ell}=\lim_{\ell \to \infty}|R_\ell(-N)|^{1/(\ell-1)}=\lim_{\ell \to \infty}|{P_\ell}(-N)|^{1/\ell}.\]
	The right hand side is within $1/N$ of $N+2$ by Lemma \ref{Chebyshev-lemma}. Taking $N=\lceil(1-\sqrt{q})^2\rceil$ produces infinitely many algebraic integers on $\left[(1-\sqrt{q})^2, (1+\sqrt{q})^2\right]$ with geometric mean of the conjugates asymptotically less than $\lceil(1-\sqrt{q})^2\rceil + 2$. Therefore $a(q) \leqslant \lceil(\sqrt{q}-1)^2 \rceil+2$. Similarly, taking $N= \lfloor(\sqrt{q}-1)^2 \rfloor-4,$ gives the estimate on $A(q)$.
\end{proof}

Proposition \ref{Chebyshev} shows that the bounds of Theorem \ref{Weil} are almost tight. We now give a simple proof of a more precise version of Proposition \ref{rho-limit} (the original claim is recovered by taking limits as $q \to \infty$).

\begin{proposition}\label{strong-rho-limit}
	Let $\rho$ be as in Definition \ref{SSS}. Then the following inequalities hold:
	\[ (q+1)^2 - 2 - q^{-2} \leqslant A(q^2) \leqslant (q+1)^2-\rho+O(q^{-1}).\]
\end{proposition}
\begin{proof}
	The first inequality follows from Proposition \ref{Chebyshev}. For the second inequality, fix a prime power $q$ and an element $\alpha \in \mathcal{A}_q'$.  For $x \in [(q-1)^2, (q+1)^2]$ Taylor expansion of $\log$ at $(q+1)^2$ gives 
	\begin{equation}\label{TaylorLog}\log x 	\leqslant 2\log (q+1) + \frac{x-(q+1)^2}{(q+1)^2} + O(q^{-3}) 
	\end{equation}
	Applying \ref{TaylorLog} over the conjugates of $\alpha$ and averaging gives the following inequality, where $\mathrm{tr}$ denotes the normalized trace (trace divided by the degree) and the implied constants do not depend on $\alpha$ or $q$:
	\[
	\log(\Norm(\alpha)^{1/\deg \alpha}) \leqslant 2 
	\log(q+1) + \frac{\mathrm{tr}\ (\alpha-(q+1)^2)}{(q+1)^2} + O({q^{-3}}).\]	
	Exponentiating and using Taylor expansion of the exponent we get
	\begin{multline*}
		A(q^2) = \limsup_{\alpha \in \mathcal{A}_q'} \Norm (\alpha)^{1/\deg{\alpha}} \leqslant \limsup_{\alpha\in \mathcal{A}_q'} (q+1)^2 e^{\mathrm{tr}(\alpha-(q+1)^2)/(q+1)^2} + O({q^{-1}})\\
		\leqslant (q+1)^2 + \limsup_{\alpha\in \mathcal{A}_q'} \mathrm{tr} (\alpha - (q+1)^2) + O(q^{-1}).
	\end{multline*}
	
	Since $(q+1)^2-\alpha$ is a totally positive algebraic integer, we have $A(q^2) \leqslant (q+1)^2 -\rho + O(q^{-1})$.
\end{proof}

The following theorem combined with Proposition \ref{Chebyshev} show that $a(q), A(q)$ can be determined up to an error of $1+q^{-1}$.

\begin{theorem}\label{AHL+}
	For all abelian varieties $A \in \mathcal{A}_q$ one of the following holds
	\begin{enumerate}
		\item The element of $\mathcal{A}_q'$ corresponding to $A$ is $\lfloor(\sqrt{q}-1)^2 \rfloor$ or $\lceil(\sqrt{q}+1)^2 \rceil$,
		\item $\lfloor(\sqrt{q}-1)^2 \rfloor + 1 \leqslant \#A(\F_q)^{1/g} \leqslant \lceil(\sqrt{q}+1)^2 \rceil -1$.	
	\end{enumerate}
\end{theorem}
\begin{proof}
	We want to estimate the norm of an algebraic integer  $\alpha \in \mathcal{A}_q'$. First we derive the lower bound. By calculus, for every integer $n>0$ the function
	\[x \mapsto \frac{x}{(x-n)^{1/(n+1)}}\]
	on the open interval $(n, +\infty)$ has minimum $n+1$ at the point $x=n+1$. Let $n\colonequals \lfloor (\sqrt{q}-1)^2\rfloor$ and let $\alpha \in \Qbar \setminus \{n\}$ be an algebraic integer such that $\GalQ \alpha \in \mathcal{A}_q'$. Let $\alpha_1,\alpha_2, ..., \alpha_d$ denote the conjugates of $\alpha$. We have
	
	\[\frac{\Norm \alpha}{ \prod_{i} (\alpha_i-n)^{1/(n+1)}} = \prod_i \frac{\alpha_i}{(\alpha_i-n)^{1/(n+1)}} \geqslant (n+1)^d.\]
	Since $\alpha$ is an algebraic integer, the product $\prod_i(\alpha_i-n)$ is a rational integer and therefore $\left|\prod_i(\alpha_i-n)\right|\geqslant 1$. Hence if $A \in \calA_q$ is an abelian variety corresponding to $\alpha$, then $\#A(\F_q)^{1/\dim A}=\left(\Norm \alpha\right)^{1/d} \geqslant (n+1)$. 
	To get an upper bound, let $N=\lceil (\sqrt{q}+1)^2 \rceil$. The function $x(N-x)^{1/(N-1)}$ has a maximum of $N-1$ at $x=N-1$. A similar argument shows that for every $\alpha \in \mathcal{A}_q'$, $\alpha \neq N$ we have $\Norm \alpha \leqslant (N-1)^d$.
\end{proof}

Theorem \ref{AHL+} improves on Theorem \ref{AHL} when $q$ is a square. In particular, combining Theorem \ref{AHL+} with Proposition \ref{Chebyshev} determines $a(q^2),A(q^2)$ up to an error of $1/2+q^{-2}$. 

It may be possible to improve Theorem \ref{AHL+} by replacing the function $x/(x-n)^{1/(n+1)}$ by a different auxiliary function. We do not know if this can be done for large $q$. However, in the next section we find better auxiliary functions for small $q$ and use them to improve the Weil estimates.

\section{Abelian varieties over small fields}

The following lemma gives a general form of the auxiliary function method used in the proof of Theorem \ref{AHL+}.

\begin{lemma}\label{auxiliary}
	Suppose that for some positive $A,B,m,M \in \R$, some monic integer polynomials $P_1, ..., P_n$, $Q_1, ..., Q_m \in \Z[x]$, and some positive  $\gamma_1, ..., \gamma_n, \beta_1, ..., \beta_m \in \R,$ the following inequalities hold for all $x \in [A,B]:$
	\[\frac{x}{\prod_i |P_i(x)|^{\gamma_i}} \geqslant m\]
	\[x\prod_j |Q_j(x)|^{\beta_j} \leqslant M.\]
	Suppose that $\alpha$ is an algebraic integer whose conjugates lie in $[A,B]$ and such that $P_i(\alpha), Q_j(\alpha) \neq 0$ for all $i,j$. Then \[m \leqslant \Norm(\alpha)^{1/\deg \alpha} \leqslant M.\]
\end{lemma}

\begin{proof}
	We will prove the lower bound, the upper bound can be derived similarly. Let $\{\alpha_1,...,\alpha_d\}$ be the Galois orbit of $\alpha=\alpha_1$. Since $P_i$ is a monic integer polynomial, the value of the product $\Pi_j|P_i(\alpha_j)|$ is a nonzero integer for every $i$. Therefore
	\[\Norm \alpha \geqslant \frac{\Norm \alpha}{\prod_{i} \prod_{j} |P_i(\alpha_j)|^{\gamma_i}}= \prod_j\left(\frac{\alpha_j}{\prod_i P_i(\alpha_j)^{\gamma_i}}\right)\geqslant m^d.\qedhere\] 
\end{proof}

We apply Lemma \ref{auxiliary} to give bounds for $a(q)$ and $A(q)$ for small values of $q$.

\begin{theorem}\label{main}
	For all but finitely many simple abelian varieties of dimension $g$ over $\F_q$ the inequalities $m(q)\leqslant \#A(\F_q)^{1/g} \leqslant M(q)$ hold, where the values of $m(q), M(q)$ are given in Table \ref{M}. The minimal polynomials $P_i, Q_j$ of elements of $\mathcal{A}_q'$ that do not satisfy the lower and the upper bound, respectively, are exactly the polynomials of Table \ref{a(q)} marked with an asterisk.\\
	\begin{table}[h]
		\begin{center}
			\begin{tabular*}{94pt}{c|cc}
				\toprule
				$q$  & $m(q)$ & $M(q)$\\
				\midrule
				2 & 1 & 4.035\\
				3 & 1.359 &5.634\\
				4 & 2.275& 7.382\\
				5 & 2.7 & 8.835\\
				7 &3.978 & 11.734\\
				8 & 4.635& 13.05\\
				9 & 5.47& 14.303\\
				\bottomrule
			\end{tabular*}\caption{Lower and upper bounds on $\#A(\F_q)^{1/g}$}\label{M} 
		\end{center}
	\end{table}
\end{theorem} 
\begin{proof}
	For every $q$ we obtain the bounds by applying Lemma \ref{auxiliary} to the segment $[(\sqrt{q}-1)^2, (\sqrt{q}+1)^2]$ and using the auxiliary polynomials $P_i, Q_j$ listed in Table \ref{a(q)}; the corresponding parameters $\gamma_i, \beta_j$ are in Table \ref{alphas}. We know explain how the auxiliary functions $P_i$ and corresponding parameters $\gamma_i$ were found, the values of $Q_j$, $\beta_j$ were obtained similarly.
	
	The auxiliary functions were found by searching for algebraic integers in $\calA_q'$ with small degree and small norm and taking $P_i$ to be equal to the corresponding minimal polynomials. The parameters $\gamma_i$ were then chosen by solving a linear programming problem of maximizing the minimum of the auxiliary function of Lemma \ref{auxiliary} on a fine mesh $S \subset [(\sqrt{q}-1)^2, (\sqrt{q}+1)^2]$. Explicitly, after $S$ and $P_i$ are fixed, the values of $\gamma_i$ were found at which \[\max_{\gamma_i} \min_{x \in S} \left(\log x - \sum \gamma_i \log P_i(x)\right)\]
	is attained. Finally, the values of $\gamma_i$ were used to create the auxiliary function $x \cdot \prod_i |P_i(x)|^{-\gamma_i}$ and find its minimum using calculus.
\end{proof}

\begin{table}[h]
	\begin{center}
		\begin{tabularx}{\textwidth}{c|Y|Y}
			\toprule
			$q$  &  $P_i \in \calA_q'$ & $Q_j \in \calA_q'$\\
			\midrule
			$2$ & n/a & \textcolor{red}{$x-5$*},   $x^2-9x+19$,   $x^3-13x^2+54x-71$\\[0.3cm]
			
			$3$ &{\textcolor{red}{$x-1$*}},   $x^2-4x+2$,   $x^3-7x^2+12x-5$ & \textcolor{red}{$x-7$*},   \textcolor{red}{$x-6$*},   $x^2-12x+34$  \\[0.3cm]
			
			$4$&\textcolor{red}{$x-1$*},  \textcolor{red}{$x-2$*},   \textcolor{red}{$x^2-5x+5$*},  $x^3-8x^2+19x-13$ &\textcolor{red}{$x-9$*},  \textcolor{red}{$x-8$*},  \textcolor{red}{$x^2-15x+55$*}, \textcolor{red}{$x^3-22x^2+159x-377$*}\\[0.7cm]
			$5$ & \textcolor{red}{$x-2$*}, \textcolor{red}{$x^2-6x+7$*}, $x^3-10x^2+28x-23$, $x^3-10x^2+30x-26$ & \textcolor{red}{$x-10$*}, \textcolor{red}{$x-9$*}, \textcolor{red}{$x^2-18x+79$*}, $x^2-17x+71$, $x^2-17x+69$\\[0.7cm]
			$7$ & \textcolor{red}{$x-3$*}, $x^2-10x+23$, $x^3-13x^2+54x-71$, $x^3-14x^2+61x-83$ & \textcolor{red}{$x-13$*}, \textcolor{red}{$x-12$*}, $x^2-23x+131$, $x^2-22x+119$, $x^3-35x^2+406x-1561$, $x^3-34x^2+381x-1405$, $x^3-34x^2+379x-1379$ \\[1.2cm]
			$8$&$x-3$, \textcolor{red}{$x-4$*}, \textcolor{red}{$x^2-9x+19$*}, $x^2-10x+23$  \textcolor{red}{$x^3-15x^2+68x-97$*}, $x^3-15x^2+71x-107$ & \textcolor{red}{$x-14$*}, \textcolor{red}{$x^2-27x+181$*}\\[0.7cm]
			
			$9$& \textcolor{red}{$x-4$*}, \textcolor{red}{$x-5$*}, $x-6$, \textcolor{red}{$x^2-11x+29$*}, $x^2-12x+33$, $x^2-12x+34$ & \textcolor{red}{$x-16$*}, \textcolor{red}{$x-15$*}, \textcolor{red}{$x^2-129x+209$*}, $x^3-43x^2+614x-2911$\\
			\bottomrule
		\end{tabularx}\caption{Auxiliary polynomials for Theorem \ref{main}.}\label{a(q)} 
	\end{center}
	
\end{table}

\begin{table}[!htbp]
	\begin{center}
		\begin{tabularx}{\linewidth}{c|Yl}
			\toprule
			$q$ & $\gamma_i$ & $\beta_j$\\
			\midrule
			$2$  & n/a& $0.141$,  $0.23$, $0.09$\\
			
			$3$  & $0.306$, $0.199$, $0.019$, $0.05$, $0.108$ & $0.1445$,  $0.155$, $0.099$\\
			
			$4$  & $0.37$,  $0.12$, $0.065$, $0.01$& $0.054$, $0.112$, $0.02$, $0.08$\\
			$5$ 
			& $0.323$, $0.063$, $0.062$, $0.007$ & $0.11$, $0.08$, $0.066$, $0.001$, $0.003$\\
			$7$  & $0.289$, $0.0048$, $0.0457$, $0.0178$ & $0.055$, $0.033$, $0.026$, $0.003$, $0.035$, $0.009$, $0.006$\\
			
			$8$ & $0.044$, $0.13$, $0.09$, $0.01$, $0.02$ & $0.08$, $0.04$\\

			$9$& $0.15$, $0.08$, $0.02$, $0.03$, $0.002$, $0.003$ & $0.033$, $0.037$, $0.033$, $0.02$ \\
			\bottomrule
		\end{tabularx}\caption{Auxiliary parameters for Theorem \ref{main}}\label{alphas} 
	\end{center}
	
\end{table}

We apply the results of Theorem \ref{main} to bound $\#A(\F_3)[2]$, the number of rational $2$-torsion points  on a simple abelian variety $A$; a similar result for Jacobians is \cite[Theorem~7.1]{Bhargava-et.al-preprint}.

\begin{corollary}
	For all but finitely many simple abelian varieties $A$ over $\F_3$ we have \[A(\F_3)[2] \leqslant 3.782^g.\] 
\end{corollary}
\begin{proof}
	Let $\Frob$ denote the $q=3$-power Frobenius endomorphism of $A$. A rational $2$-torsion point is in the kernel of both $\Frob+1$ and $\Frob-1$. Therefore \[\#A(\F_3)[2] \leqslant \sqrt{\deg(\Frob-1)\deg(\Frob+1)} = \deg (\Frob^2-1)^{1/2}=\#A(\F_{3^2})^{1/2}.\] Applying Theorem \ref{main} to the right hand side proves the claimed inequality.
\end{proof}
Given a field extension $L/K$ and a scheme $X$ over $K$, a point $x \in X(L)$ is called \slantsf{new} if $x \not\in X(F)$ for all fields $F$ with $K \subseteq F \subsetneq L$.
\begin{corollary}
	Let $A/\F_q$ be a simple abelian variety. Then $A$ has no new points over $\F_{q^r}$ if and only if one of the following holds:
	\begin{enumerate}
		\item $r=2$, $q \in \{2,3,4\}$, and $A$ is the quadratic twist of an abelian variety $A'$ with $A'(\F_q)=0$ (these are described in Theorem \ref{main} for $q=3,4$, and in \cite{Madan-Pal1977} for $q=2$),
		
		\item\label{case2} $r=3$, $q=2$, and the element of $\calA_q'$ corresponding to $A$ is $4$ or $5$.
		
	\end{enumerate}
\end{corollary}
\begin{proof}
	Suppose $r=2$. The equality $A(\F_{q^2})=A(\F_q)$ is equivalent to the assertion that the quadratic twist $A'$ of $A$ has a unique rational point over $\F_q$. From now on suppose $r>2$.
	
	The Weil conjectures imply that $A$ has a new point when $q$ and $r$ are sufficiently large, as we will now show. Suppose that for some abelian variety $A/\F_q$ of dimension $g$ the set $A(\F_{q^r})$ has no new points. Then the following inequalities hold
	\[\left(q^{r/2}-1\right)^{2g} \leqslant \#A(\F_{q^r}) \leqslant \sum_{d|r, \, d<r} \#A(\F_{q^d}) \leqslant \sum_{d|r, \, d<r} \left(q^{d/2} + 1\right) ^{2g} \leqslant 2 \sqrt{r} \left(q^{r/4}+1\right)^{2g},\] because the number of divisors of $r$ is at most $2\sqrt{r}$.
	So $(q^{r/2} -1)^{2g} \leqslant 2\sqrt{r}(q^{r/4}+1)^{2g}$, which implies $(q^{r/4}-1)^{2g} \leqslant 2\sqrt{r}$. The last inequality together with the condition $r>2$ imply that the pair $(q,r)$ is equal to one of the following $(2, 3),$ $(2,4),$ $(2,5),$ $(2,6),$ $(3,3),$ $(3,4),$ or $(4,3).$ Out of these only the pairs $(2,3), (2,4)$ satisfy the inequality $(q^{r/2}-1)^{2g} \leqslant \sum_{d|r, d<r} (q^{d/2}+1)^{2g}$ for some $g\geqslant 1$. Suppose $q=2, r=3$, then $A(\F_8)=A(\F_2)$. Since $A$ is simple, $A_{\F_8}$ is isotypic: $A_{\F_8} \sim B^e$ with $e \leqslant 3$. Therefore, by Theorem \ref{main} the following inequality holds with finitely many exceptions $\#A(\F_2)^{1/g} \leqslant 4.04 < 4.635 \leqslant \#A(\F_8)^{1/g}$. In each of the exceptional cases we test if $A(\F_2)=A(\F_8)$ by a direct computation; the resulting exceptions are listed in the statement of Case \ref{case2}. Suppose that $q=2$ and $r=4$, which means $A(\F_{16})=A(\F_2)$. Then $A(\F_4)=A(\F_{16})$, and so $A_{\F_4}$ is the quadratic twist of an abelian variety $A'$ with $A'(\F_4)=0$. Theorem \ref{main} implies that the element of $\calA_4'$ corresponding to $A'$ is $1$. A direct computation shows that the element of $\calA_2'$ corresponding to $A$ is $3$ and that for this abelian variety $A(\F_2) \neq A(\F_{16})$.
\end{proof}

\section*{Acknowledgements} 

I thank my advisor Bjorn Poonen for many helpful suggestions. I thank Dinesh Thakur for suggesting the reference \cite{Madan-Pal1977}. Finally I thank Padma Srinivasan, Nicholas Triantafillou, and Dmitri Kubrak for useful comments and discussions.

\end{document}